\documentclass[10pt,twoside]{amsart}
\usepackage{amsmath}
\usepackage{amssymb}
\usepackage[arrow, matrix, curve]{xy}
\usepackage{hyperref}
\usepackage{url}
\usepackage{color}

\numberwithin{equation}{section}
\numberwithin{table}{section} 

\theoremstyle{plain}

\newtheorem{thm}{Theorem}[section]
\newtheorem{thm*}{Theorem}
\newtheorem{lem}[thm]{Lemma}

\newtheorem{cor}[thm]{Corollary}

\theoremstyle{definition}

\newtheorem{defi}[thm]{Definition}
\newtheorem{exa}[thm]{Example}
\newtheorem{situation}[thm]{Situation}

\newtheorem{rem}[thm]{Remark}

\newtheorem{problem}[thm]{Problem}

\newcommand{\mm}{\mathfrak m}

\newcommand{\F}{\mathbb{F}}
\newcommand{\N}{\mathbb{N}}
\newcommand{\Prim}{\mathbb{P}}
\newcommand{\Q}{\mathbb{Q}}
\newcommand{\Z}{\mathbb{Z}}

\newcommand{\Cc}{\mathcal{C}}
\newcommand{\Ec}{\mathcal{E}}
\newcommand{\Fc}{\mathcal{F}}
\newcommand{\Lc}{\mathcal{L}}
\newcommand{\Oc}{\mathcal{O}}
\newcommand{\Sc}{\mathcal{S}}

\newcommand{\shM}{\mathcal{M}}

\newcommand{\Th}{\text{h}}

\DeclareMathOperator{\chara}{char}
\DeclareMathOperator{\dirsum}{\oplus}
\DeclareMathOperator{\coKer}{coker}

\DeclareMathOperator{\Dim}{dim}
\DeclareMathOperator{\Spec}{Spec}
\DeclareMathOperator{\lK}{H}
\DeclareMathOperator{\Syz}{Syz}
\DeclareMathOperator{\Proj}{Proj}
\DeclareMathOperator{\Deg}{deg}
\DeclareMathOperator{\Rank}{rank}
\DeclareMathOperator{\HKF}{HK}
\DeclareMathOperator{\HKM}{e_{HK}}
\DeclareMathOperator{\sg}{syzgap}
\DeclareMathOperator{\frob}{F}

\newcommand{\fpb}[1]{\frob^{#1\ast}}
\newcommand{\ra}{\rightarrow}
\newcommand{\lra}{\longrightarrow}

\newcommand{\qpot}{^{[q]}}

\newcommand{\lto}{\longrightarrow}

\newcommand  {\fl}[1]      {\lfloor #1 \rfloor}

\begin{document}



\author{Daniel Brinkmann and Almar Kaid}

\address{Universit\"at Osnabr\"uck, Fachbereich 6: Mathematik/Informatik, Albrechtstr. 28a,
49069 Osnabr\"uck, Germany}

\email{dabrinkm@uni-osnabrueck.de and akaid@uni-osnabrueck.de}


\title[Rank-$2$ syzygy bundles on Fermat curves]{Rank-$2$ syzygy bundles on Fermat curves and an application to Hilbert-Kunz functions}

\date{\today}

\begin{abstract}
In this paper we describe the Frobenius pull-backs of the syzygy bundles $\Syz_C(X^a,Y^a,Z^a)$, $a\geq 1$, on the projective Fermat curve $C$ of degree $n$ in characteristics coprime to $n$, either by giving their strong Harder-Narasimhan filtration if $\Syz_C(X^a,Y^a,Z^a)$ is not strongly semistable or in the strongly semistable case by their periodicity behavior. 
Moreover, we apply these results to Hilbert-Kunz functions, to find Frobenius periodicities of the restricted cotangent bundle $\Omega_{\Prim^2}|_C$ of arbitrary length and a problem of Brenner regarding primes with strongly semi\-stable reduction.
\end{abstract}

\maketitle

\vspace*{-0.7cm}

\section*{Introduction}
In this paper we study rank-$2$ syzygy bundles for certain monomial ideals, namely the ideals $(X^a,Y^a,Z^a)$ for $a \geq 1$ on a smooth projective Fermat curve
$C:=\Proj(k[X,Y,Z]/(X^n+Y^n+Z^n))$. We focus mainly on positive characteristic, where semistability is not preserved by the Frobenius morphism. The first example 
is the restricted cotangent bundle $\Omega_{\Prim^2}|_C$ from the projective plane which can be identified with the syzygy bundle $\Syz_C(X,Y,Z)$.
For this bundle semistability is well known (if $n\geq 2$) but no complete answer is known regarding strong semistability, i.e., for what characteristics and curve degrees is $\fpb{e}(\Omega_{\Prim^2}|_C)$ semistable for all $e \geq 0$. 

The bundles $\Syz_C(X^a,Y^a,Z^a)$ on Fermat curves exhibit important arithmetic properties of vector
bundles in general. For instance, Brenner proved in \cite{brennerstronglysemistable} with this setup that there exists no restriction theorem for strong semistability of Bogomolov type. In \cite{miyaoka} he used the syzygy bundle $\Syz_C(X^2,Y^2,Z^2)$ on the Fermat quintic to show that strong semistability is not an open property in arithmetic deformations which was conjectured
by N. I. Shepherd-Barron. In \cite{brennerkaidfrobeniusdescent} Brenner and the second author disproved a conjecture of Joshi, which is related to the Grothendieck $p$-curvature conjecture, exploiting the arithmetic properties of Fermat curves and syzygy bundles. Despite these papers, there has been no complete treatment of 
the semistability properties of the bundles $\Syz_C(X^a,Y^a,Z^a)$ on Fermat curves.

Fermat rings are also popular examples for characteristic $p$ methods in commutative algebra, in particular \textit{Hilbert-Kunz theory} and the theory of \textit{tight closure}. Monsky and Han exploited Fermat rings to obtain explicit computations for the 
Hilbert-Kunz multiplicity and the Hilbert-Kunz function (see \cite{monskyhan}, \cite{handiss}, \cite{mason} and \cite{irred}), which we partially complement in this paper. Due to work of Brenner \cite{bredim2} and Trivedi \cite{tri1} there is a strong relationship to strongly semistable vector bundles which is our main tool to compute Hilbert-Kunz functions in this paper. We remark that for the arithmetic behavior of tight closure the ideals
$(X^a,Y^a,Z^a)$ which are topic of this paper are useful examples (see \cite{brennerkatzmanarithmetic}, \cite{singhcomputation}).

In Section \ref{sec:hkviavb} we recall the necessary notations of Hilbert-Kunz theory and vector bundles and discuss their relationship mentioned above for syzygy bundles of rank $2$. Afterwards, we turn our attention to \textit{syzygy gaps} and how they appear in Han's and Monsky's computations \cite{monskyhan} of Hilbert-Kunz multiplicities of ideals generated by a fixed (positive) power of the variables in two-dimensional Fermat rings.

Section \ref{sec:periodicity} is devoted to the case where the bundle $\Syz_C(X^a,Y^a,Z^a)$, $a\geq 1$, is strongly semistable. We will show how work of Kustin, Rahmati and Vraciu \cite{vraciu}
can be applied in this situation to obtain a periodicity up to twist of the Frobenius pull-backs of these bundles as well as a sharp bound for the length of this periodicity 
(cf. Theorem \ref{frobpernontri}). This periodic behavior enables us to compute the Hilbert-Kunz functions of the ideals $(X^a,Y^a,Z^a)$, $a\geq 1$, in the 
Fermat rings $k[X,Y,Z]/(X^n+Y^n+Z^n)$ (cf. Theorem \ref{hkfsss}) and allows us to generalize a theorem of Brenner and the second author 
(cf. Theorem \ref{p+-1}).

In Section \ref{sec:possibleperiods} we will use the methods developed in Section \ref{sec:periodicity} to construct for every odd prime $p$ and every natural number $l$ an $n$ such that 
the Frobenius pull-backs of $\Omega_{\Prim^2}|_C$ admit up to twist a periodicity of length $l$, where $C$ is the projective Fermat curve of degree $n$ (see Example 4.1). 
Moreover, we will see in Example 4.2 that in characteristic two, the Frobenius pull-backs of the restricted cotangent bundle admit a periodicity up to twist 
if and only if $n=3$.

In Section \ref{nonstronglysemistablesection} we investigate the case where $\Syz_C(X^a,Y^a,Z^a)$ is not strongly semi\-stable. In this case we can explicitly describe the minimal $e_0$ such 
that the $e_0$-th Frobenius pull-back of this bundle has a strong Harder-Narasimhan filtration. Moreover, this filtration is explicitly computed 
(cf. Theorem \ref{almarfilt}). This result is used to complete the computations of Hilbert-Kunz functions from Section \ref{sec:periodicity} (cf. Corollary \ref{almarhkf}). The explicit computation of the entire Hilbert-Kunz function
for ideals $(X^a,Y^a,Z^a)$ with $a > 1$ is new (at least to the best knowledge of the authors of this paper). 

Finally, in Section \ref{miyaokasection} we deal with a special instance of the Miyaoka problem proposed by Brenner, i.e., the question how the property of $\Sc:=\Syz_C(X^a,Y^a,Z^a)$ being strongly semistable depends on the parameter $a$, 
the characteristic $p$ and the degree $n$ of the projective Fermat curve $C$. Theorem \ref{fermatgenericfibersemistability} will show that $\Sc$ is semistable 
in characteristic zero if and only if it is strongly semistable in all characteristics $p\equiv \pm 1$ modulo $2n$. From this result we will deduce that 
the Harder-Narasimhan filtration of $\Sc$ in characteristic zero can be computed via reduction modulo $p$ (cf. Theorem \ref{hnfiltrationcharzero}).

We remark that most of the results in this paper can easily be translated into algorithms that are suitable for a computer algebra system (e.g., \cite{CocoaSystem}).

We would like to thank Holger Brenner for many useful discussions and comments. We also thank Axel St\"abler for his corrections to this paper. The results of Sections \ref{sec:periodicity} and \ref{sec:possibleperiods} 
are contained in Chapter 7 of the first author's PhD thesis \cite{drdaniel} and those of Sections \ref{nonstronglysemistablesection} and \ref{miyaokasection} are 
contained in Chapter 4 of the PhD thesis \cite{almar} of the second author.

\section{Preliminaries}
\label{sec:hkviavb}
Let $(R,\mm)$ be a local Noetherian ring of characteristic $p>0$ and dimension $d$. Let $I=(f_1,\ldots,f_m)$ be an $\mm$-primary ideal. Let $\lambda_R$ 
denote the length function for $R$-modules. The function 
$$\HKF(I,p^e):\N\lto\N, ~ e\longmapsto\lambda_R\left(R/\left(f_1^{p^e},\ldots,f_m^{p^e}\right)\right)$$
is called the \textit{Hilbert-Kunz function} of $I$. We call the limit
$$\lim_{e\ra\infty}\frac{\HKF(I,p^e)}{p^{ed}},$$
whose existence was proven by Monsky in \cite{hkmexists}, the \textit{Hilbert-Kunz multiplicity} of $I$. We denote this limit by $\HKM(I)$ and call 
$\HKM(R):=\HKM(\mm)$ the Hilbert-Kunz multiplicity of $R$. Sometimes we will use the symbol $(f_1,\ldots,f_m)^{[p^e]}$ to denote the ideal 
$(f_1^{p^e},\ldots,f_m^{p^e})$.

Throughout this paper we are mainly interested in the graded situation, where $k$ is an algebraically closed field and $R$ a standard-graded normal $k$-algebra of dimension $2$. Moreover, we consider homogeneous elements $f_1,f_2,f_3 \in R$ with $\Deg(f_1) = \Deg(f_2) = \Deg(f_3)=  a \geq 1$ such that the ideal $I:=(f_1, f_2, f_3) \subset R$ is $R_+$-primary. These elements give rise to the short exact (presenting) sequences
$$0\lra\Syz_C(f_1,f_2,f_3)(m)\lra\bigoplus_{i=1}^3\Oc_C(m-a)\stackrel{f_1,f_2,f_3}{\lra}\Oc_C(m)\lra 0,$$
for all $m \in \Z$ on the smooth projective curve $C:=\Proj(R)$. The kernel sheaf $\Syz_C(f_1,f_2,f_3)$ is locally free  of rank $2$ and is called the \textit{syzygy bundle} for $f_1,f_2,f_3$.
In the case $\chara(k)=p>0$ the Hilbert-Kunz function of the ideal $I$ can be computed via the formula
\begin{align}\label{eq:hkfgeom}
\begin{split}
\Dim_k\left(R/\left(f_1^{p^e},f_2^{p^e},f_3^{p^e}\right)\right)_m & = ~~ \Th^0(C,\Oc_C(m))-\sum_{i=1}^3\Th^0(C,\Oc_C(m-p^ea)) \\
 & +\Th^0\left(C,\Syz_C\left(f_1^{p^e},f_2^{p^e},f_3^{p^e}\right)(m)\right)
\end{split}
\end{align}
and summation over all $m\in\Z$ (see for instance \cite{bredim2}).
If $\Sc$ is a vector bundle on a smooth projective curve over an algebraically closed field $k$ we define its \textit{slope} by the quotient 
$\mu(\Sc):=\frac{\Deg(\Sc)}{\Rank(\Sc)}$ with $\Deg(\Sc):=\Deg(\bigwedge^{\Rank{\Sc}}(\Sc))$. We recall that $\Ec$ is \textit{semistable} if for every non-trivial locally free subsheaf $\Fc$ the inequality $\mu(\Fc) \leq \mu(\Ec)$ holds. The bundle $\Ec$ is \textit{stable} if this inequality is always strict. Hence the rank-$2$ syzygy bundle $ \Sc:=\Syz_C(f_1,f_2,f_3)$  is not semistable if and only if there exists a line bundle $0 \neq \Lc \subset \Ec$ with $\Deg(\Lc) > \mu(\Sc) = \Deg(\Sc)/2 = -3an/2$, where $n = \Deg(C)$. For such a line bundle of maximal degree this filtration constitutes the so-called \textit{Harder-Narasimhan filtration} (or \textit{HN-filtration}) of $\Sc$ and the line bundle $\Lc$ is the \textit{(maximal) destabilizing subbundle}. We often write the HN-filtration of a rank-$2$ vector bundle as a short exact sequence $0 \ra \Lc \ra \Sc \ra \shM \ra 0$, where the quotient $\shM$ is a line bundle with $\Deg(\shM) < \Deg(\Lc)$. Note that concept and existence of a HN-filtration is 
trivial 
for bundles of rank $2$ but for bundles of higher rank it is much more complicated (see \cite{hnfilt}). 

In positive characteristic $p$ we consider the \textit{absolute Frobenius morphism} $\frob: C \ra C$ which is the identity on the curve $C$ and the $p$-th power map on the structure sheaf $\Oc_C$. It is well-known that the Frobenius pull-back $\fpb{}(\Sc)$ of a semistable vector bundle $\Sc$  is in general not semistable
(see for instance \cite[Example 3.2]{amplevboncurves} for Serre's counter example). If $\fpb{e}(\Sc)$ is semistable for all $e\geq 0$ then $\Sc$ is 
called \textit{strongly semistable}. This notion is due to Miyaoka (cf. \cite[Section 5]{defsss}). In the case of a rank-$2$ vector bundle 
the HN-filtration $0 \ra \Lc \ra \fpb{e}(\Sc) \ra \shM \ra 0$ of  a (non-semistable) Frobenius pull-back $\fpb{e}(\Sc)$ is called the 
\textit{strong HN-filtration} of the bundle $\Sc$. We indicate that, as for the HN-filtration itself, the concept and existence of a strong 
HN-filtration is highly non-trivial for vector bundles of higher rank (see \cite{langer} for a detailed account).

Due to the work of Brenner \cite{bredim2} and Trivedi \cite{tri1} strongly semistable vector bundles are closely related to Hilbert-Kunz 
theory. We only state this connection for the rank-$2$ vector bundles which we consider in this paper.

\begin{thm}
\label{planecurvehkformula}
Let $C=V_+(G) \subset \Prim^2$ denote a smooth plane curve of degree $n$ over an algebraically closed field $k$ of positive characteristic $p$. 
Let $R=k[X,Y,Z]/(G)$ be its homogeneous coordinate ring and let $I=(f_1,f_2,f_3)$ denote a homogeneous
$R_+$-primary ideal such that $\deg(f_1)=\deg(f_2)=\deg(f_3)=a$. Let $\Sc:=\Syz_C(f_1,f_2,f_3)$. Then the following hold.
\begin{enumerate}
\item The bundle $\Sc$ is strongly semistable if and only if $\HKM(I) = \frac{3n}{4} a^2$.
\item If $\Sc$ is not strongly semistable with strong Harder-Narasimhan filtration $0 \ra \Lc \ra \fpb{e}(\Sc) \ra \shM \ra 0$, then 
$$\HKM(I)=n\left(\left(\frac{\Deg(\Lc)}{np^e}+\frac{3a}{2}\right)^2+\frac{3a^2}{4}\right).$$
Moreover, we have
$$\Deg(\Lc)=-\frac{3anp^e}{2} + np^e\cdot\sqrt{-\frac{3a^2}{4}+\frac{\HKM(I)}{n}}\in\left(-\frac{3anp^e}{2},-anp^e\right].$$
\item If $\Sc$ is not semistable, then the Hilbert-Kunz multiplicity of $I$ equals
$$\HKM(I) = \frac{3n}{4} a^2+ \frac{\ell^2}{4n},$$
where $0< \ell \leq na$ is an integer with $\ell \equiv an \!\!\mod 2$.
\item If the bundle $\Sc$ is semistable, but not strongly semistable, then the Hilbert-Kunz multiplicity of $I$
equals
$$\HKM(I) =\frac{3n}{4} a^2+ \frac{\ell^2}{4np^{2e}},$$
where $e\geq 1$ is the number such that $\fpb{e-1}(\Sc)$ is semistable and $\fpb{e}(\Sc)$ is not semistable and
$0< \ell \leq n(n-3)$ is an integer with $\ell \equiv pna \!\!\mod 2$.
\end{enumerate}
\end{thm}

\begin{proof}
For the proof of (1) and (2) see \cite[proof of Corollary 4.6]{bredim2}. For the proof of (3) and (4) see \cite[proof of Theorem 5.3]{tri1}. In all three cases the quoted proofs only consider the case $a = 1$ but they are easy to generalize. For an explicit proof see \cite[proofs of Corollary 1.4.9 and Proposition 1.4.11]{almar}.  
\end{proof}
Next, we discuss syzygy gaps and how they appear in the computation of Hilbert-Kunz multiplicities. The final result of Han and Monsky (cf. Theorem \ref{hkmfermat}) will combine the 
case by case description of the Hilbert-Kunz multiplicities of the ideals $(X^a,Y^a,Z^a)$ from the last theorem. Moreover, Theorem \ref{hkmfermat} will give us a numerical criterion to check the strong semistability of the sheafs $\Syz_C(X^a,Y^a,Z^a)$.

\begin{defi}
Let $R:=k[X,Y]$ and let $f_1$, $f_2$, $f_3\in R$ be non-zero and homogeneous. By Hilbert's Syzygy Theorem we have a splitting 
$\Syz_R(f_1,f_2,f_3)\cong R(-\alpha)\oplus R(-\beta)$ 
for some $\beta\leq \alpha \in\N$. We call the difference $\alpha-\beta$ the \textit{syzygy gap of $f_1$, $f_2$, $f_3$} and denote it by
$\sg(f_1,f_2,f_3)$.
In the special case $f_1=X^a$, $f_2=Y^b$, $f_3=(X+Y)^c$ for some positive integers $a$, $b$, $c$, we denote $\sg(X^a,Y^b,(X+Y)^c)$ by $\delta(a,b,c)$.
\end{defi}

One can show that $\delta$ extends to a unique continuous function $\delta^{\ast}$ on $[0,\infty)^3$ with the property 
$|\delta(t)-\delta(s)|\leq \left\Vert t-s\right\Vert_1$. Moreover, if the underlying polynomial ring is defined over a field of 
positive characteristic $p$, the extension of $\delta$ satisfies 
$\delta^{\ast}\left(\frac{t}{p}\right)=p^{-1}\cdot\delta^{\ast}(t).$ 
In what follows we will not distinguish between $\delta$ and $\delta^{\ast}$. Our next goal is to explain how $\delta$ can be computed. 
The reference is Han's thesis \cite{handiss} resp. \cite{mason} for an alternative proof. 

We are interested in the \emph{taxicab distance} of elements of the form $\tfrac{t}{p^s}$ to the set 
$$L_{\text{odd}}:=\left\{u\in\N^3|u_1+u_2+u_3\text{ is odd}\right\},$$
where $s$ is an integer and $t$ a three-tuple of non-negative real numbers.
Note that for given $\tfrac{t}{p^s}$ there is at most one $u\in L_{\text{odd}}$ satisfying $\left\Vert \frac{t}{p^s}-u\right\Vert_1<1.$ 
Moreover, the only candidates for $u_i$ are the rounding ups and rounding downs of $\tfrac{t_i}{p^s}$.

\begin{thm}[Han]\label{hansthm}
Let $t=(t_1,t_2,t_3)\in[0,\infty)^3.$ If the $t_i$ do not satisfy the strict triangle inequality (w.l.o.g. $t_1\geq t_2+t_3$), we have $\delta(t)=t_1-t_2-t_3.$
If the $t_i$ satisfy the strict triangle inequality and there are $s\in \Z$, $u\in L_{odd}$ with $\left\Vert p^st-u\right\Vert_1<1$, then there is such a pair $(s,u)$ with minimal $s$ and with this pair $(s,u)$ we get
$$\delta(t)=\frac{1}{p^s}\cdot\left(1-\left\Vert p^st-u\right\Vert_1\right).$$ Otherwise, one has $\delta(t)=0$.
\end{thm}

With the help of the $\delta$-function we can state the following theorem.

\begin{thm}[Han, Monsky]\label{hkmfermat} The Hilbert-Kunz multiplicity of an ideal $I:=(X^a,Y^a,Z^a)$, $a\geq 1$, of the Fermat ring $k[X,Y,Z]/(X^n+Y^n+Z^n)$ equals
$$\frac{3a^2n}{4}+\frac{n^3}{4}\cdot\delta\left(\frac{a}{n},\frac{a}{n},\frac{a}{n}\right)^2.$$
\end{thm}

\begin{proof}
The case $a=1$ is due to Han (cf. \cite[Theorem 2.30]{handiss}) and was generalized by Monsky (cf. \cite[Theorem 2.3]{irred}). 
\end{proof}

Combining Theorem \ref{planecurvehkformula}(1) and Theorem \ref{hkmfermat} one obtains the following numerical criterion for strong semistability.

\begin{cor}\label{deltastrongsemistable} Let $\gcd(p,n)=1$.
The bundle $\Syz_{V_+(X^n+Y^n+Z^n)}(X^a,Y^a,Z^a)$ is strongly semistable if and only if 
$\delta\left(\frac{a}{n},\frac{a}{n},\frac{a}{n}\right)=0.$
\end{cor}


\begin{exa}
\label{syzsquarestableexample}
In \cite[Remark 2]{miyaoka} Brenner asks whether the syzygy bundle $\Syz_C(X^2,Y^2,Z^2)$ is strongly semistable on the Fermat curve $C$ of degree 
$n$ in characteristics $p \equiv \pm 1 \!\! \mod n$. Using Corollary \ref{deltastrongsemistable} we are able to answer this question positively for 
the interesting cases $n \geq 3$. Since $2p^s \equiv \pm 2 \!\! \mod 2n$ for every $s \geq 0$, we see that the distance of $\frac{2p^s}{n}$ to the 
next odd integer is $\frac{n-2}{n}$ and $3 \frac{n-2}{n}= 3- \frac{6}{n} \geq 1$. Thus, $\delta(\frac{2}{n},\frac{2}{n},\frac{2}{n})=0$ and 
$\HKM((X^2,Y^2,Z^2))=\frac{12}{n}$, which is equivalent to the strong semistability of the syzygy bundle $\Syz_C(X^2,Y^2,Z^2)$. In the case $n=1$ we 
have $\Syz_C(X^2,Y^2,Z^2) \cong \Oc_{\Prim^1}(-3) \oplus \Oc_{\Prim^1}(-3)$ and this bundle is obviously strongly semistable. For $n=2$ the curve 
equation $X^2+Y^2+Z^2=0$ yields a non-trivial global section of total degree $2$. Since $\deg(\Syz_C(X^2,Y^2,Z^2)(2))=-4$, the bundle is not semistable.
\end{exa}

\section{The strongly semistable case}
\label{sec:periodicity}
We start by fixing the notations which will be used for the rest of this paper.

\begin{situation}\label{situation}
Let $k$ be an algebraically closed field of characteristic $p>0$ and let $R:=k[X,Y,Z]/(X^n+Y^n+Z^n)$ for $n\geq 1$. We denote the projective Fermat curve 
of degree $n$ by $C:=\Proj(R)$. We assume $\gcd(p,n)=1$, hence $C$ is a smooth curve. We will consider the ideal $I:=(X^a,Y^a,Z^a)$, $a\geq 1$, 
as well as its first syzygy module $\Syz_R(I)$ and the bundle $\Syz_C(I)$. 
\end{situation}

An important special case in the situation above is $a=1$, since $\Syz_C(X,Y,Z)$ is isomorphic to 
the restricted cotangent bundle $\Omega_{\Prim^2}|_C$ which can be seen by using the Euler sequence. Note that the restricted 
cotangent bundle $\Omega_{\Prim^2}|_C$ is (at least) semistable on every smooth plane curve of degree $\geq 2$ (cf. \cite[Corollary 3.5]{tri1}).

In this section we deal with the case where the syzygy bundle $\Syz_C(X^a,Y^a,Z^a)$ is strongly semistable. For this purpose we will use results of 
\cite{vraciu} where the authors study the minimal projective resolution of the quotients $R/\left(X^a,Y^a,Z^a\right)$ depending on $a$ and $n$.

We start with the special case where some Frobenius pull-back of $\Syz_C(X^a,Y^a,Z^a)$ becomes a direct sum of twisted copies of the structure sheaf. Hence we have an isomorphism
$$\Syz_C(X^{aq},Y^{aq},Z^{aq})\cong\Oc_C(l)^2$$
for some $q=p^e$ and some $l\in\Z$ due to the semistability of the pull-back. Note that this may happen only if $p=2$ or if $a$ is even as one sees by comparing the degrees of the bundles above. 

By \cite[Theorem 2.1]{claudia} the Hilbert-Kunz function is given by $\HKM(I^{[e_0]})\cdot p^{2(e-e_0)}$, where $e_0$ is (minimal) such that 
$R/I^{[e_0]}$ has finite projective dimension.

The minimal $e$ such that $\Syz_C(X^{ap^e},Y^{ap^e},Z^{ap^e}))$ is a direct sum of twisted copies of the structure sheaf can be computed with the following theorem.

\begin{thm}[Kustin, Rahmati, Vraciu]\label{finprojdim}
In Situation \ref{situation}, the quotient $R/I$ has finite projective dimension as an $R$-module if and only if one of the following conditions is satisfied

\begin{enumerate}
 \item $n|a$,
 \item $p=2$ and $n\leq a$,
 \item $p$ is odd and there exist positive integers $J$ and $e$ with $J$ odd such that
$$\left|Jp^e-\frac{a}{n}\right|<\begin{cases}3^{e-1} & \text{if }p=3,\\ \frac{p^e-1}{3} & \text{if }p^e\equiv 1\text{ mod }3,\\ \frac{p^e+1}{3} & \text{if }p^e\equiv 2\text{ mod }3.\end{cases}$$
\end{enumerate}
\end{thm}

\begin{proof}
See \cite[Theorem 6.3]{vraciu}.
\end{proof}

In \cite{li} the author uses this theorem to study the interaction between strong semistability of $\Syz_C(I)$, the projective dimesions of 
$R/I\qpot$ and the diagonal $\frob$-threshold of $I$.

In the case where $\Syz_C(X^a,Y^a,Z^a)$ is strongly semistable but no Frobenius pull-back of it splits as a direct sum of twisted copies of the structure 
sheaf, we have that all quotients $R/(X^{aq},Y^{aq},Z^{aq})$ have infinite projective dimension as $R$-module, because the first syzygy modules 
$\Syz_R(X^{aq},Y^{aq},Z^{aq})$ are not free. In this case we can use a result of Kustin, Rahmati and Vraciu to obtain a twisted Frobenius periodicity of 
$\Syz_C(X^a,Y^a,Z^a)$ in the following sense.

\begin{defi}
Let $S$ be a normal, standard-graded $k$-domain of dimension two, where $k$ is an algebraically closed field of characteristic $p>0$. Let $\Ec$ be a vector bundle 
over the smooth projective curve $Y:=\Proj(S)$. Assume there are $0\leq s<t\in\N$ such that the Frobenius pull-backs $\fpb{e}(\Ec)$ of $\Ec$ (and all their twists) 
are pairwise non-isomorphic for $0\leq e\leq t-1$ and $\fpb{t}(\Ec)\cong\fpb{s}(\Ec)(m)$ holds for some $m \in \Z$. We say that $\Ec$ admits a 
\textit{twisted $(s,t)$-Frobenius periodicity.} The bundle $\Ec$ admits a \textit{twisted Frobenius periodicity} if there are $0\leq s<t\in\N$ such that 
$\Ec$ admits a twisted $(s,t)$-Frobenius periodicity.
\end{defi}

To state the necessary result of Kustin, et al., we need one more definition.

\begin{defi}
Let $r$ and $s$ be positive integers with $r+s=n$. Then we define 
$$\phi_{r,s}:=\begin{pmatrix}0 & Z^r & -Y^r & X^s\\ -Z^r & 0 & X^r & Y^s\\ Y^r & -X^r & 0 & Z^s\\ -X^s & -Y^s & -Z^s & 0\end{pmatrix}.$$
\end{defi}

\begin{thm}[Kustin, Rahmati, Vraciu]\label{inftyfree}
In Situation \ref{situation}, let $a=\theta\cdot n+r$ with $\theta\in\N$ and $r\in\{1,\ldots,n-1\}$. Assume that $Q:=R/I$ 
has infinite projective dimension. If $\theta=2\cdot\eta-1$, then the homogenous minimal free resolution of $Q$ is given by
$$\xymatrix{
\ldots \ar[r]^{\phi_{r,n-r}} & F_1(-n) \ar[r]^{\phi_{n-r,r}} & F_2 \ar[r]^{\phi_{r,n-r}} & F_1\ar[r] & R(-a)^3\ar[r] & R
}$$
with free graded modules 
\begin{align*}
F_1 & :=R(-3\eta n+n-r)^3 \oplus R(-3\eta n+2n-3r),\\
F_2 & :=R(-3\eta n+n-2r)^3 \oplus R(-3\eta n).
\end{align*}
If $\theta=2\cdot\eta$, then the homogenous minimal free resolution of $Q$ is given by
$$\xymatrix{
\ldots \ar[r]^{\phi_{n-r,r}} & F_1(-n) \ar[r]^{\phi_{r,n-r}} & F_2 \ar[r]^{\phi_{n-r,r}} & F_1 \ar[r] & R(-a)^3\ar[r] & R
},$$
where the first three free graded modules are defined as
\begin{align*}
F_1 & :=R(-3\eta n-2r)^3 \oplus R(-3\eta n-n),\\
F_2 & :=R(-3\eta n-n-r)^3 \oplus R(-3\eta n-3r)
\end{align*}
\end{thm}

\begin{proof}
The statement follows from \cite[Theorem 3.5]{vraciu} combined with \cite[Theorems 5.14 and 6.1]{vraciu}.
\end{proof}

\begin{cor}\label{inftysyz}
Under the hypothesis of Theorem \ref{inftyfree}, we have
$$\begin{aligned}
\Syz_R\left(X^a,Y^a,Z^a\right)(m) & \cong \begin{cases}\coKer(\phi_{n-r,r}) & \text{if }\theta\text{ is even,}\\ \coKer(\phi_{r,n-r}) & \text{if }\theta\text{ is odd,}\end{cases}\\
    & \cong \begin{cases}\Syz_R(X^r,Y^r,Z^r) & \text{if }\theta\text{ is even,}\\ \Syz_R(X^{n-r},Y^{n-r},Z^{n-r}) & \text{if }\theta\text{ is odd}\end{cases}
\end{aligned}$$
for some $m\in\Z$.
\end{cor}

\begin{proof}
The first isomorphism is clear from the free resolution. The second isomorphism follows from the free resolution by considering the case $a=r$.
\end{proof}

At this point it is not clear whether the modules $M_r=\Syz_R(X^r,Y^r,Z^r)$ for $1\leq r\leq n-1$ are pairwise non-isomorphic or not. A criterion 
to decide this is given by the Hilbert-series, which can be computed with the help of the next theorem, which is a slight but useful improvement 
of \cite[Lemma 1.1]{miyaoka}. Geometrically spoken, it considers a smooth projective curve 
$$D:= V_+(Z^n -F(X,Y)) \subset \Prim^2= \Proj k[X,Y,Z],$$
where $F(X,Y) \in k[X,Y]$ denotes a homogeneous polynomial of degree $n$, and relates the sheaves $\Syz_D(X^{a},Y^{b},Z^{c})$ to the sheaves 
$\Syz_D(X^a,Y^{b},F(X,Y)^i)$ which come from $\Prim^1$ via the Noetherian normalization $D \ra \Prim^1=\Proj k[X,Y]$. We will use the following result several 
times in the sequel of this paper. 

\begin{thm}\label{ses}
Let $k$ be a field, $S:=k[X,Y,Z]/(Z^n-F(X,Y))$ and $F\in k[X,Y]$ homogeneous of degree $n\geq 2$. Let $a, b, c\geq 1$ and write $c=n\cdot q +r$ 
with $0\leq r\leq n-1$ and $q\in\N$. For all $s\in\Z$ we have a short exact sequence
\begin{align*}
0 & \lra \Syz_S(X^a,Y^b,Z^{c+n-2r})(s-r)\\
 & \stackrel{\psi}{\lra} \Syz_S(X^a,Y^b,F^q)(s-r)\dirsum\Syz_S(X^a,Y^b,F^{q+1})(s)\\
 & \stackrel{\phi}{\lra} \Syz_S(X^a,Y^b,Z^c)(s)\lra 0,
\end{align*}
where the maps are defined via 
\begin{align*}
\psi(h_1,h_2,h_3):= & ((Z^{n-r}\cdot h_1,h_2,h_3),(-h_1,-Z^r\cdot h_2,-Z^r\cdot h_3))\\
\phi((f_1,f_2,f_3),(g_1,g_2,g_3)):= & (f_1+Z^{n-r}\cdot g_1,Z^r\cdot f_2+g_2,Z^r\cdot f_3+g_3). 
\end{align*}
\end{thm}

\begin{proof}
The injectivity of $\psi$ is clear and the exactness at the middle spot is straightforward. The proof that $\phi$ is surjective can be found 
in \cite[Lemma 2.1]{holgeralmar}. See also \cite[Chapter 4]{drdaniel} for a detailed proof and a generalization.
\end{proof}

\begin{thm}\label{HilbSer} The notations are the same as in Theorem \ref{ses}. For $l\in\N$ we use the abbreviation $\Sc_l:=\Syz_S(X^a,Y^b,Z^l)$. 
Then the Hilbert-series of $\Sc_c=\Syz_S(X^a,Y^b,Z^c)$ is given by
\begin{equation*}
\lK_{\Sc_c}(t) = \frac{(t^{r}-t^{n})\cdot \lK_{\Sc_{c-r}}(t)+(1-t^{r})\lK_{\Sc_{c+n-r}}(t)}{1-t^{n}}.
\end{equation*}
\end{thm}

\begin{proof}
 Let $c':=c+n-2r=nq+n-r$ and $r':=n-r$. Since $c'+n-2r'=c$, Theorem \ref{ses} yields
\begin{align*}
 \lK_{\Sc_c}(t) & = t^{r}\lK_{\Sc_{c-r}}(t)+\lK_{\Sc_{c+n-r}}(t)-t^{r}\lK_{\Sc_{c'}}(t)\quad\text{and}\\
 \lK_{\Sc_{c'}}(t) & = t^{r'}\lK_{\Sc_{c-r}}(t)+\lK_{\Sc_{c+n-r}}(t)-t^{r'}\lK_{\Sc_c}(t).
\end{align*}
Substituting $\lK_{\Sc_{c'}}(t)$ in the first formula and solving for $\lK_{\Sc_c}(t)$ gives the result.
\end{proof}

Turning back to the question of computing the Hilbert-series of the $R$-modules $M_r=\Syz_R(X^r,Y^r,Z^r),$ $1\leq r\leq n-1,$ we obtain via 
Theorem \ref{HilbSer}
\begin{align}\label{hilbsermr}
\nonumber \lK_{M_r}(t) &= \frac{(t^r-t^n)\cdot \lK_{\Syz_R\left(1,Y^r,Z^r\right)}(t)+(1-t^r)\lK_{\Syz_R\left(-Y^n-Z^n,Y^r,Z^r\right)}(t)}{1-t^n}\\
\nonumber &= \frac{2t^r(t^r-t^n)+(1-t^r)(t^n+t^{2r})}{(1-t)^3}\\
&=\frac{t^n+3t^{2r}-3t^{n+r}-t^{3r}}{(1-t)^3}.
\end{align}
Since the Hilbert-series of $M_r$ and $M_s$ for $1\leq r<s\leq n-1$ are not multiples of each other, the modules $(M_r)_{1\leq r\leq n-1}$ are pairwise 
non-isomorphic.

Now we are able to prove the following.

\begin{thm}\label{frobpernontri}
Assume we are in Situation \ref{situation}. The bundle $\Syz_C(X^a,Y^a,Z^a)$ admits a twisted Frobenius periodicity if and only if 
$\delta\left(\tfrac{a}{n},\tfrac{a}{n},\tfrac{a}{n}\right)=0.$ Moreover, the length of this periodicity is bounded from above by the order of $p$ in $\Z/(2n)$.  
\end{thm}

\begin{proof}
This follows from Corollary \ref{deltastrongsemistable} and because the isomorphism class of the module 
$\Syz_R(X^{aq},Y^{aq},Z^{aq})$ depends on $r\equiv aq\text{ }(n)$ and the parity of $\tfrac{aq-r}{n}$.
\end{proof}

\begin{rem}
One should mention that one already knew that strongly semistable bundles admit a twisted Frobenius periodicity. This is because the coefficients 
of the equation $X^n+Y^n+Z^n=0$ lie in a finite field and therefore all moduli spaces are finite dimensional varieties over $\F_p$. Hence, the number of 
$\F_p$-rational points gives an upper bound for the length of the periodicity but it's very rough and there is no hint how to compute it explicitly.
\end{rem}

The next example shows that the upper bound for the length of the periodicity from Theorem \ref{frobpernontri} is the best possible.

\begin{exa}\label{fermatexa4}
Let $p=37$ and $n=14$. Then $p^e=14\cdot \theta +r$ with 
$$(\theta,r)=\begin{cases}(\text{even},9) & \text{if }e\equiv 1\text{ }(3),\\ (\text{odd},11) & \text{if }e\equiv 2\text{ }(3),\\ (\text{even},1) & \text{if }e\equiv 0\text{ }(3).\end{cases}$$
From this computation it is easy to see that $\delta(\tfrac{1}{14},\tfrac{1}{14},\tfrac{1}{14})=0$, hence $\Omega_{\Prim^2}|_C$ is strongly semistable and 
we get a twisted $(0,3)$-Frobenius periodicity 
$$\fpb{3}(\Omega_{\Prim^2}|_C)\cong\Omega_{\Prim^2}|_C\left(-\tfrac{3}{2}\cdot (q-1)\right).$$
\end{exa}

Theorem \ref{frobpernontri} enables us to compute the Hilbert-Kunz function of $(X^a,Y^a,Z^a)$, if the bundle $\Syz_C(X^a,Y^a,Z^a)$ is strongly semistable 
and none of its Frobenius pull-backs split as a direct sum of twisted copies of the structure sheaf.

\begin{thm}\label{hkfsss}
Assume we are in Situation \ref{situation} and that $\Syz_C(I)$ is strongly semistable and none of its Frobenius pull-backs are a direct sum of twisted copies 
of the structure sheaf. Let $aq=ap^e=n\theta+r$ with $\theta\in\N$ and $0\leq r<n$. Then 
$$\HKF(I,q)= \left\{\begin{aligned} & \frac{3a^2n}{4}\cdot q^2-\frac{3n}{4}r^2+r^3 && \text{if }\theta\text{ is even,}\\ & \frac{3a^2n}{4}\cdot q^2-\frac{3n}{4}(n-r)^2+(n-r)^3 && \text{if }\theta\text{ is odd.}\end{aligned}\right.$$
\end{thm}

\begin{proof}
Under our assumptions on $\Syz_C(I)$, the quotients $R/(X^{aq},Y^{aq},Z^{aq})$ have infinite projective dimension and their resolutions are given by 
Theorem \ref{inftyfree}. By Corollary \ref{inftysyz} we have an isomorphism $\Syz_R(I\qpot)\cong\Syz_R(J)(l)$, where depending on $\theta$, we use $J$ 
and $b$ to denote either $(X^r,Y^r,Z^r)$ and $r$ or $(X^{n-r},Y^{n-r},Z^{n-r})$ and $n-r$. We obtain a diagram with exact rows
$$\xymatrix{
0\ar[r] & \Syz_C(I\qpot)(m)\ar[r]\ar[d]^{\cong} & \Oc_C(m-ap^e)^3\ar[r] & \Oc_C(m)\ar[r] & 0\\
0\ar[r] & \Syz_C(J)(m+l)\ar[r] & \Oc_C(m+l-b)^3\ar[r] & \Oc_C(m+l)\ar[r] & 0.
}$$
Taking global sections as in Equation \eqref{eq:hkfgeom} the claim follows by using 
$$\Th^0\left(C,\Syz_C\left(X^{ap^e},Y^{ap^e},Z^{ap^e}\right)(m)\right)=\Th^0\left(C,\Syz_C\left(X^b,Y^b,Z^b\right)(m+l)\right)$$
and a computation similar to that in the proof of \cite[Corollary 4.1]{holgeralmar}.
\end{proof}

\begin{exa}
Let $p=37$ and $n=14$. By Example \ref{fermatexa4} we obtain the Hilbert-Kunz function 
$$\HKF(R,37^e)=\left\{\begin{aligned} & \frac{21}{2}\cdot 37^{2e}-\frac{243}{2} && \text{if }e\equiv 1\text{ }(3),\\ & \frac{21}{2}\cdot 37^{2e}-\frac{135}{2} && \text{if }e\equiv 2\text{ }(3),\\ & \frac{21}{2}\cdot 37^{2e}-\frac{19}{2} && \text{if }e\equiv 0\text{ }(3).\end{aligned}\right.$$
\end{exa}

In \cite{holgeralmar} the authors proved that $\Omega_{\Prim^2}|_C$ admits a twisted $(0,1)$-Frobenius periodicity if $p\equiv -1\text{ }(2n)$. 
The authors tried to adopt their proof to the case $p\equiv 1\text{ }(2n)$ but failed (cf. \cite[Remark 3.5]{holgeralmar}).

\begin{thm}\label{p+-1}
Assume $p\equiv \pm 1\text{ }(2n)$ in Situation \ref{situation}. Then $\Omega_{\Prim^2}|_C$ is strongly semistable with twisted $(0,1)$-Frobenius periodicity
$$\fpb{e}(\Omega_{\Prim^2}|_C)\cong\Omega_{\Prim^2}|_C\left(-\frac{3}{2}\cdot(p^e-1)\right).$$
Moreover, the Hilbert-Kunz function of $R$ is given by
$$\HKF(R,p^e)=\frac{3n}{4}\cdot p^{2e}+1-\frac{3n}{4}.$$
\end{thm}

\begin{proof}
As $\delta(\tfrac{1}{n},\tfrac{1}{n},\tfrac{1}{n})=0$, the bundle $\Omega_{\Prim^2}|_C$ is strongly semistable. Since $p$ has to be odd by assumption 
and $a=1$, the Frobenius pull-backs of $\Omega_{\Prim^2}|_C$ are not of the form $\Oc_C(l)^2$ as mentioned at the beginning of this section. 
If $p\equiv 1\text{ }(2n)$ then all powers $p^e$ can be written in the form even$\cdot n+1$ and if 
$p\equiv -1\text{ }(2n)$ all powers $p^{2e}$ can be written as even$\cdot n+1$ and the powers $p^{2e+1}$ can be written as odd$\cdot n+n-1$. 
The periodicity of $\Syz_C(X,Y,Z)$ follows from Corollary \ref{inftysyz} and the statement about the Hilbert-Kunz function is due to Theorem \ref{hkfsss}.
\end{proof}

\begin{rem}\label{frobperdegzero}
Note that one can construct $(s,t)$-Frobenius periodicities in the classical sense, e.g. of degree zero bundles, from the twisted $(s,t)$-Frobenius periodicities 
obtained from Theorem \ref{frobpernontri} as follows: 

If $a$ is even, one can consider the bundle $\Syz_C(X^a,Y^a,Z^a)(\tfrac{3a}{2})$.

If $a$ is odd, one obtains a $(s,t)$-Frobenius periodicity of $\Syz_D(U^{2a},V^{2a},W^{2a})(3a)$ on the Fermat curve $D$ of degree $2n$ as it was done in \cite[Example 5.1]{holgeralmar}.



Recall that due to Lange and Stuhler the vector bundles of degree zero admitting a $(0,t)$-Frobenius periodicity are exactly those that are \'{e}tale trivializable 
(cf. \cite{langestuhler}). For the syzygy bundles which only admit a $(s,t)$-Frobenius periodicity with $s \geq 1$ there is only a finite trivialization, namely the composition
of the Frobenius morphism and a suitable \'{e}tale covering.

By Theorem \ref{p+-1} and this remark, the syzygy bundle $\Syz_D(U^2,V^2,W^2)(3)$ admits a $(0,1)$-Frobenius 
periodicity on the projective Fermat curve $D$ of degree $2n$, provided $p\equiv\pm 1$ modulo $2n$. St\"abler computed the \'{e}tale map trivializing this 
bundle in characteristics $p\equiv -1$ modulo $2n$ explicitly in \cite{axel}.
\end{rem}

\section{Which Frobenius periodicities can be achieved?}
\label{sec:possibleperiods}
The next examples deal with the question which twisted $(0,t)$-Frobenius periodicities of the bundles $\Omega_{\Prim^2}|_C$ can be achieved. A sufficient condition for 
having a twisted Frobenius periodicity is (cf. Theorem \ref{frobpernontri}) $\delta\left(\tfrac{1}{n},\tfrac{1}{n},\tfrac{1}{n}\right)=0,$ 
which is equivalent to the condition that the distances of all triples $v_e:=\bigl(\tfrac{p^e}{n},\tfrac{p^e}{n},\tfrac{p^e}{n}\bigr)$ to 
$L_{\text{odd}}$ are at least one. Let $p^e=\theta_e\cdot n+r_e$ with $\theta_e,r_e\in\N$ and $0\leq r_e<n$. The nearest element to $v_e$ in 
$L_{\text{odd}}$, which potentially has a taxicab distance $<1$ from $v_e$, is given by the component-wise rounding ups of $v_e$ if $\theta_e$ is even and by the component-wise rounding downs of $v_e$ if 
$\theta_e$ is odd. This leads to the (sufficient) conditions 
$$\begin{aligned}
&\left.\begin{aligned}
& 3\cdot \left(1-\tfrac{r_e}{n}\right)\geq 1 && \text{if }\theta_e\text{ is even,}\\
& 3\cdot \tfrac{r_e}{n}\geq 1 && \text{if }\theta_e\text{ is odd.}
\end{aligned}\right\}
&\Longleftrightarrow &
\left\{\begin{aligned}
& 2\cdot n\geq 3\cdot r_e && \text{if }\theta_e\text{ is even,}\\
& 3\cdot r_e\geq n && \text{if }\theta_e\text{ is odd.}
\end{aligned}\right.
\end{aligned}$$

\begin{exa}\label{fermatexa5}
Let $p$ be odd, $l\in\N$ and $n=\tfrac{p^{l+1}+1}{2}$. For $0\leq e\leq 2l+2$ we have 
\begin{equation}\label{respe}
p^e=\left\{
\begin{aligned}
& 0\cdot n+p^e && \text{if }0\leq e\leq l,\\
& \left(2\cdot p^{e-l-1}-1\right)\cdot n+n-p^{e-l-1} && \text{if }l+1\leq e\leq 2l+1,\\
& \left(2\cdot p^{l+1}-2\right)\cdot n+1 && \text{if }e= 2l+2.
\end{aligned}\right.
\end{equation}
This shows that $p^e$ is of the form $\text{even}\cdot n+p^{e'}$ or $\text{odd}\cdot n+n-p^{e'}$ for some $0\leq e'\leq l$.
Since $2n=p^{l+1}+1\geq 3p^{e'}$ for all $0\leq e'\leq l$, we see that $\delta\left(\tfrac{1}{n},\tfrac{1}{n},\tfrac{1}{n}\right)=0$.
By Corollary \ref{inftysyz} we find that $\fpb{e}(\Omega_{\Prim^2}|_C)$ is isomorphic to
$$\fpb{e'}(\Omega_{\Prim^2}|_C)\left(-\frac{3}{2}\cdot (p^e-p^{e'})\right),$$
where $e'\equiv e\text{ mod }l+1\text{ with }e'\in\{0,\ldots,l\}.$ With Equation (\ref{hilbsermr}) we find a twisted $(0,l+1)$-Frobenius 
periodicity. The Hilbert-Kunz function is given by
$$\HKF(R,p^e)=\frac{3n}{4}\cdot\left(p^{2e}-p^{2e'}\right)+p^{3e'},$$
where again $e'\equiv e\text{ mod }l+1\text{ with }e'\in\{0,\ldots,l\}.$
\end{exa}

\begin{exa}\label{fermatexa7}
Let $p=2$, $1\leq l\in\N$ and $n$ odd with $2^l<n<2^{l+1}$, hence $n=2^l+x$ with $0<x<2^l$. We want to show that $\Omega_{\Prim^2}|_C$ admits a 
twisted Frobenius periodicity if and only if $n=3$. The condition $\delta\left(\tfrac{1}{n},\tfrac{1}{n},\tfrac{1}{n}\right)=0$ forces
$$2n=2^{l+1}+2x\geq 3\cdot 2^l\geq 3\cdot 2^e$$ 
for all $0\leq e\leq l$. This is equivalent to $x\geq 2^{l-1}$. Since $2^{l+1}=n+2^l-x$ with $0<2^l-x<2^l<n$, we need that the inequality
$3\cdot (2^l-x)\geq n$ holds. This is equivalent to $x\leq 2^{l-1}$. This shows that a twisted Frobenius periodicity might appear only in the case 
$x=2^{l-1}$ resp. $n=3\cdot 2^{l-1}$. Since $n$ is odd, we obtain the single possibility $n=3$.
Now let $n=3$. Then $\Omega_{\Prim^2}|_C$ is strongly semistable since $\delta(\tfrac{1}{3},\tfrac{1}{3},\tfrac{1}{3})=0$. We remark that this is well-known: since $C$ is an elliptic curve and $\Omega_{\Prim^2}|_C$ is semistable the strong semistability follows also from \cite[Theorem 2.1]{mehtaramanathanhomogeneous}.
\end{exa}

\section{The non-strongly semistable case}
\label{nonstronglysemistablesection}
In this section we want to compute the Hilbert-Kunz function of $I=(X^a,Y^a,Z^a)$ under the condition that the syzygy bundle $\Syz_C(X^a,Y^a,Z^a)$ 
is \emph{not} strongly semi\-stable on the Fermat curve $C$. We do this by an explicit computation of the strong HN-filtration of these bundles.

\begin{lem}
\label{nonstablefrobenius}
In addition to Situation \ref{situation} let $b$ denote a natural number and write $b=nl+r$ with $0 \leq r \leq n-1$.
\begin{enumerate}
\item If $l$ is even and $r > \frac{2n}{3}$, then $\Syz_C(X^b,Y^b,Z^b)$ is
not semistable on $C$.
Moreover, one has $$\Gamma(C,\Syz_C(X^b,Y^b,Z^b)(n(l+1+\frac{l}{2}))) \neq 0.$$
\item If $l$ is odd and $r<\frac{n}{3}$, then $\Syz_C(X^b,Y^b,Z^b)$ is not
semistable on $C$. Moreover,
one has $$\Gamma(C,\Syz_C(X^b,Y^b,Z^b)(n(l+\fl{\frac{l}{2}})+3r)) \neq 0.$$
\end{enumerate}
\end{lem}

\begin{proof}
For the proof of $(1)$ see \cite[Proposition 1]{miyaoka}. For the proof of part $(2)$ which
is rather similar to $(1)$ see \cite[Lemma 4.2.8(2)]{almar}.
\end{proof}

We list the following two neat consequences of the previous lemma.

\begin{cor}
\label{pmdegreenonstable}
In Situation \ref{situation} assume  $p \equiv n \pm 1 \!\! \mod 2n$ for an even natural number
$n \geq 4$. Then $\fpb{}(\Omega_{\Prim^2}|_C)$ is not semistable on $C$.
\end{cor}

\begin{proof}
For $p \equiv n + 1 \!\! \mod 2n$ this follows from part $(1)$ of Lemma \ref{nonstablefrobenius} and for the case $p \equiv n - 1 \!\! \mod 2n$ one applies part $(2)$.
\end{proof}

\begin{cor}
\label{acongruentdcase}
If in Situation \ref{situation} we have $a \equiv n \!\! \mod 2n$, then $\Syz_C(X^a,Y^a,Z^a)$ is not semistable on $C$.
\end{cor}

\begin{proof}
This follows immediately from part $(2)$ of Lemma \ref{nonstablefrobenius}.
\end{proof}

Now, using Lemma \ref{nonstablefrobenius} we are able to prove that the $s$th Frobenius pull-back (where $s$ is the integer in Han's Theorem \ref{hansthm}) 
is not semistable. Moreover, we explicitly compute a strong Harder-Narasimhan filtration of $\Syz_C(X^a,Y^a,Z^a)$.

\begin{thm}\label{almarfilt}
In Situation \ref{situation} assume $p>\tfrac{3a}{2n}$. If $\delta\left(\tfrac{a}{n},\tfrac{a}{n},\tfrac{a}{n}\right)\neq 0$ and $s$ is the integer of 
Han's Theorem \ref{hansthm}, then $\fpb{s}(\Syz_C(X^a,Y^a,Z^a))$ is not semistable. Let $ap^s=nl+r$ with $l\in\N$ and $0\leq r\leq n-1$. Then the 
Frobenius pull-back $\fpb{s}(\Syz_C(X^a,Y^a,Z^a))$ has a strong Harder-Narasimhan filtration given by
$$0\ra \Oc_C(-m)\ra\fpb{s}(\Syz_C(X^a,Y^a,Z^a))\ra \Oc_C\left(m-3ap^s\right)\ra 0,$$
with $m=n(l+1+\tfrac{l}{2})$ if $l$ is even and $m=n(l+\lfloor\tfrac{l}{2}\rfloor)+3r$ if $l$ is odd. Moreover, this filtration is minimal 
for $p\geq n-3$ and $s\geq 1$ in the sense that $\fpb{(s-1)}(\Syz_C(X^a,Y^a,Z^a))$ is semistable.
\end{thm}

\begin{proof}
Since $\Sc:=\Syz_C(X^a,Y^a,Z^a)$ is not strongly semistable, we know by Han's Theorem \ref{hansthm} and Corollary \ref{deltastrongsemistable} that the taxicab distance from $(\frac{ap^s}{n},\frac{ap^s}{n},\frac{ap^s}{n})$ to the nearest element in $L_{\text{odd}}$ is $<1$. We have $s \geq 0$ due to the assumption $p>\frac{3a}{2n}$. We compute the taxicab distance in dependence on $l$. First, we consider the case where $l$ is even. So the distance from $\frac{ap^s}{n}$ to the nearest odd integer is $\frac{n-r}{n}$ and the taxicab distance to the closest element in $L_{\text{odd}}$ is $3 \frac{n-r}{n}=3-\frac{3r}{n}$, which is by assumption $<1$. Hence, we obtain $r>\frac{2n}{3}$. Now we apply Lemma \ref{nonstablefrobenius}(1) to $b=ap^s$ and see that $\fpb{s}(\Sc)\cong \Syz_C(X^{ap^s},Y^{ap^s},Z^{ap^s})$ is not semistable. Moreover, we have a non-trivial mapping
$\Oc_C(-n(l+1+\frac{l}{2})) \lto \fpb{s}(\Sc)$. We want to prove that this mapping constitutes the HN-filtration of the $s$th Frobenius pull-back, i.e., the mapping has no zeros on $C$. First, we compute the Hilbert-Kunz multiplicity of the ideal $I:=(X^a,Y^a,Z^a)$ in the ring $R$. Since
$$\delta \left(\frac{a}{n},\frac{a}{n},\frac{a}{n}\right) = \frac{1}{p^s} \left(1 - \left(3 - \frac{3r}{n}\right) \right) = \frac{1}{p^s} \left(\frac{3r-2n}{n}\right)$$ we obtain
$$\HKM(I)=\frac{3n}{4}a^2+ \frac{n^3}{4}\left(\frac{1}{p^s} \left(\frac{3r-2n}{n}\right)\right)^2=\frac{3n}{4}a^2+\frac{(3r-2n)^2}{4p^{2s}}\cdot n$$
via Theorem \ref{hkmfermat}. Now we use the Hilbert-Kunz multiplicity $\HKM(I)$ to read off the degree of the destabilizing invertible sheaf $\Lc \subset \fpb{s}(\Sc)$. Theorem \ref{planecurvehkformula}(2) yields
$$\deg(\Lc) = -np^s \left(\frac{3a}{2}-\sqrt{\frac{(3r-2n)^2}{4p^{2s}}}\right) = -n^2(l+1+\frac{l}{2}).$$
Thus $\deg(\Lc \otimes \Oc_C(n(l+1+\frac{l}{2})))=0$. Since this line bundle does have a non-trivial section, we obtain $\Lc \cong \Oc_C(-n(l+1+\frac{l}{2}))$ and the HN-filtration is indeed
$$0 \subset \Oc_C(-n(l+1+\frac{l}{2})) \subset \fpb{s}(\Sc).$$
The assertion on the quotient line bundle is clear since the determinant bundle is additive on short exact sequences.

The case that $l$ is odd follows essentially in the same way and we omit it here.
Finally, we prove the assertion about the minimality of the HN-filtration. Assume the minimal integer $e$ such that 
$\fpb{e}(\Sc)$ is not semistable is strictly smaller than $s$. If $l$ is even, we obtain the equality
$\ell=n(3r-2n)p^{e-s}$, where the integer $\ell$ is defined as in Theorem \ref{planecurvehkformula}(4). But this equality can only hold for prime numbers $p \leq n-3$ since $0 < 3r-2n \leq n-3$ (we have $r \leq n-1$ and $p \nmid n$).

If $l$ is odd, we have $\ell=n(n-3r)p^{e-s}$. If $r\geq 1$, we can conclude as above. In case of $r=0$, we see that $p^s$ has to divide $l$ and we can conclude that $a=nb$ with $b$ odd. But this means that the minimal $s$ of Han's Theorem is actually $s=0$ which contradicts the assumption $s \geq 1$.
\end{proof}

If the characteristic is sufficiently large, Theorem \ref{almarfilt} also yields a numerical criterion for semistability of the syzygy bundle $\Syz_C(X^a,Y^a,Z^a)$. 
Semistability of these bundles in characteristic $0$ is part of Section \ref{miyaokasection}.

\begin{cor}
\label{semistabilitycriterionbigchar}
Assume the situation of Theorem \ref{almarfilt}. If $\chara(k)\geq n-3$, then $\Syz_C(X^a,Y^a,Z^a)$ is semistable if and only if $s>0$.
\end{cor}

\begin{proof}
This is an immediate consequence of Theorem \ref{almarfilt}.
\end{proof}

\begin{exa}
Let $p \neq 3$ and consider the Fermat cubic $C$ ($n=3$), which is an elliptic curve.  We recall that by 
\cite[Theorem 2.1]{mehtaramanathanhomogeneous} the syzygy bundle $\Syz_C(X^a,Y^a,Z^a)$ is semistable on $C$ if and only if it is strongly semistable, 
which is by Corollary \ref{deltastrongsemistable} equivalent to $\delta\left(\frac{a}{n},\frac{a}{n},\frac{a}{n} \right) = 0$. 
If $p>\frac{a}{2}$, it follows from Theorem \ref{almarfilt} that this bundle is not semistable on $C$ if and only if $a=3l$ for some odd integer $l$. 
This description for non-semistability does not hold when $p \leq \frac{a}{2}$. For instance, let $p=2$ and $a=3\cdot 2+1=7$. Then the taxicab 
distance from $(\frac{7}{6},\frac{7}{6},\frac{7}{6})$ to $(1,1,1) \in L_{\text{odd}}$ equals $\frac{3}{6}<1$ and hence 
$\delta\left(\frac{7}{3},\frac{7}{3},\frac{7}{3} \right) \neq 0$.
\end{exa}

Combining Theorem \ref{almarfilt} with \cite[Theorem 2.1]{claudia}, one obtains the Hilbert-Kunz functions of the ideals $(X^a,Y^a,Z^a)$ for $q = p^e \gg 0$. 
Note that the term $\HKM(I)$ can be explictly computed via Theorem \ref{hkmfermat}.

\begin{cor}\label{almarhkf}
In the situation of Theorem \ref{almarfilt} one has 
\begin{equation}\label{eq:hkf}\HKF(I,p^e) = \HKM(I)p^{2e}\text{ for all } e \gg 0.\end{equation}
\end{cor}



\begin{exa}\label{p3n7}
Let $p=3$ and $n=7$. Then $\delta\left(\tfrac{1}{7},\tfrac{1}{7},\tfrac{1}{7}\right)=\tfrac{1}{63}$ with $s=2$. By Theorem \ref{almarfilt} we have the 
strong Harder-Narasimhan filtration
$$0\ra\Oc_C(-13)\ra\fpb{2}(\Syz_C(X,Y,Z))\ra\Oc_C(-14)\ra 0.$$
Since $-3^t+4<0\Leftrightarrow t\geq 2$, the $e$-th Frobenius pull-backs of $\Syz_C(X,Y,Z)$ split as 
$$\Oc_C\left(-13\cdot p^{e-2}\right)\oplus\Oc_C\left(-14\cdot p^{e-2}\right)$$
for $e\geq 4$. This gives
$$\HKF(R,3^e)=\HKM(R)\cdot 3^e=\frac{427}{81}\cdot 3^e$$
for $e\geq 4$. The other values are $\HKF(R,3^0)=1$, $\HKF(R,3^1)=27$, $\HKF(R,3^2)=419$ and $\HKF(R,3^3)=3843$ by an explicit computation with CoCoA 
\cite{CocoaSystem}. Moreover, another explicit computation shows
\begin{align*}
\Syz_R(X^9,Y^9,Z^9) &\cong\Syz_R(X^5,Y^5,Z^5)(-6) \text{ and}\\
\Syz_R(X^{27},Y^{27},Z^{27}) &\cong R(-39)\oplus R(-42).
\end{align*}
\end{exa}

\begin{rem}
We want to relate our results on the Hilbert-Kunz functions of the ideals $(X^a,Y^a,Z^a)$ with those of other authors. 
As a consequence of a result of Brenner \cite[Theorem 6.1]{bredim2func}, the Hilbert-Kunz function of an $(X,Y,Z)$-primary ideal $I$ of the ring 
$R=k[X,Y,Z]/(X^n+Y^n+Z^n)$ has the form
\begin{align}\label{eq:holgersphi}
e\mapsto \HKM(I)\cdot p^{2e}+\phi(p,e),
\end{align}
where $\phi(p,\_)$ is an eventually periodic function. By Equation \eqref{eq:hkf}, we see that $\phi(p,e)=0$ 
for all large $e$ and all $p$ coprime to $n$ if some Frobenius pull-back of $\Syz_C(I)$ splits as a direct sum of twisted structure sheaves. If this is 
not the case, we obtain from Theorem \ref{hkfsss} that the shape of $\phi(p,\_)$ does only depend on the residue class of $p$ modulo $2n$. All in all, 
we have seen that for every fixed $n$ there are only finitely many possibilities for $\phi(p,\_)$.
\end{rem}

\section{Strongly semistable reduction on the relative Fermat curve}
\label{miyaokasection}
In this section we deal with a problem proposed by Brenner in \cite[Problem 5]{brennerstronglysemistable} which contains a special case of Miyaoka's problem \cite[Problem 5.4]{defsss}. 

\begin{problem}[Brenner]
\label{brennerfermatproblem}
How does the strong semistability of $\Syz_C(X^a,Y^a,Z^a)$ on the Fermat curve $C$ of degree $n$ depend on the characteristic $p$, the degree $n$ and the integer $a$? In particular, for fixed $n$ and $a$, is the set of prime numbers such that 
$\Syz_C(X^a,Y^a,Z^a)$ is strongly semistable finite, infinite or does it contain almost all prime numbers?
\end{problem}

We have already answered the first part by Theorem \ref{hkmfermat} and Corollary \ref{deltastrongsemistable} in Section \ref{sec:hkviavb} (cf. also Example \ref{syzsquarestableexample}). The following theorem gives a numerical criterion for semistability of the syzygy bundle $\Syz_C(X^a,Y^a,Z^a)$ on a Fermat curve in characteristic $0$. It also shows that if the syzygy bundle is semistable in characteristic $0$, then it has strongly semistable reduction for infinitely many prime numbers. 
 Before we state the theorem, we recall some notation for a relative curve $\Cc \ra \Spec \Z$. For a prime number $p$ we denote by $\Cc_p:= \Cc \times_{\Spec R} \Spec \F_p$ the \emph{special fiber} over the closed point $(p) \in \Spec \Z$ and by $\Cc_0:= \Cc \times_{\Spec R} \Spec \Q$ the \emph{generic fiber} over the generic point $(0) \in \Spec \Z$. Finally, we recall that by a theorem of Dirichlet there exist infinitely many prime numbers in any arithmetic progression.

\begin{thm}
\label{fermatgenericfibersemistability}
Let $a \geq 1$  be an integer, and consider the smooth projective relative Fermat curve
$\Cc := \Proj(\Z_n[X,Y,Z]/(X^n+Y^n+Z^n)) \lto \Spec \Z_n$. Write $a=nl+r$ with $0 \leq r < n$ and let $\tilde{r} \equiv a \!\! \mod 2n$ $(0 \leq \tilde{r} <2n)$.
Then the following conditions are equivalent:
\begin{enumerate}
\item One has $r \leq \frac{2n}{3}$ if $l$ is even and $r \geq \frac{n}{3}$ if $l$ is odd.
\item One has $\tilde{r} \leq \frac{2n}{3}$ if $\tilde{r}<n$ and $\tilde{r} \geq \frac{4n}{3}$ if $\tilde{r} \geq n$.
\item For all prime numbers $p>\max\{n-3, \frac{3a}{2n}\}$ the integer $s$ of Han's Theorem \ref{hansthm} is either $\geq 1$ or $\delta(\frac{a}{n},\frac{a}{n},\frac{a}{n})=0$.
\item The syzygy bundle $\Syz_{\Cc_p}(X^a,Y^a,Z^a)$ is strongly semistable on the special fiber $\Cc_p$ for all prime numbers $p \equiv \pm 1 \!\! \mod 2n$ with $p>\frac{3a}{2n}$.
\item The syzygy bundle $\Syz_{\Cc_0}(X^a,Y^a,Z^a)$ is semistable on the generic fiber $\Cc_0$.
\end{enumerate}
\end{thm}

\begin{proof}
$(1) \Leftrightarrow (2)$. This is obvious.\\
$(1) \Rightarrow (3)$. Assume there is a prime number $p > \frac{3a}{2n}$ such that $s=0$. Then $\Syz_{\Cc_p}(X^a,Y^a,Z^a)$ is not semistable on the fiber $\Cc_p$ by Theorem \ref{almarfilt}. The proof of Theorem \ref{almarfilt} shows that we have either $r>\frac{2n}{3}$ or $r < \frac{n}{3}$ depending on $l$ . But this contradicts the assumption.\\
$(3) \Rightarrow (1)$. Suppose the assumption on $r$ in (1) does not hold. Then the bundle $\Syz_{\Cc_p}(X^a,Y^a,Z^a)$ is non-semistable on every fiber $\Cc_p$ by Lemma
\ref{nonstablefrobenius}. But this contradicts Theorem \ref{almarfilt} since $s \geq 1$ and $\fpb{s}(\Syz_{\Cc_p}(X^a,Y^a,Z^a))$ is the first non-semistable Frobenius pull-back in characteristics $p \geq n-3$.\\
$(2) \Rightarrow (4)$. We have $p^sa \equiv \pm \tilde{r} \!\! \mod 2n$ for every prime number $p \equiv \pm 1 \!\! \mod 2n$ and every $s\geq 0$.
If $\tilde{r} <n$, then the distance from $\frac{p^s}{n}$ to the nearest odd integer is $\frac{n-\tilde{r}}{n}$. Hence the taxicab distance from $(\frac{p^sa}{n},\frac{p^sa}{n},\frac{p^sa}{n})$ to the nearest element in $L_{\text{odd}}$ is
$$3 \frac{n- \tilde{r}}{n}=3-\frac{3 \tilde{r}}{n} \geq 3 - \frac{2n}{n}=1.$$
Similarly, if $\tilde{r} \geq n$, then the distance from $\frac{p^sa}{n}$ to the nearest odd integer
is $\frac{\tilde{r}-n}{n}$, and thus the taxicab distance to the nearest element in $L_{\text{odd}}$ is
$$3 \frac{\tilde{r}-n}{n}=\frac{3\tilde{r}}{n}-3 \geq \frac{4n}{n}-3=1.$$
Hence $\delta(\frac{a}{n},\frac{a}{n},\frac{a}{n})=0$  in characteristics $p \equiv \pm 1 \!\! \mod 2n$ (with $p>\frac{3a}{2n}$) and the syzygy bundle $\Syz_{\Cc_p}(X^a,Y^a,Z^a)$ is strongly semistable by Corollary \ref{deltastrongsemistable}.\\
$(4) \Rightarrow (5)$. This follows from the openness of semistability (see \cite[paragraph after Proposition 5.2]{defsss}.\\
$(5) \Rightarrow (1)$. This follows immediately from Lemma \ref{nonstablefrobenius}.
\end{proof}

\begin{rem}
To determine the semistability of the bundles $\Syz_C(X^a,Y^a,Z^a)$ one might also consider the application of restriction theorems. It is well-known that the vector bundle $\Sc:=\Syz_{\Prim^2}(X^a,Y^a,Z^a)$ is stable on the projective plane with Chern classes $c_1(\Sc)=-3a$ and $c_2(\Sc)=3a^2$ (see for instance \cite[Corollary 3.2]{brennerlookingstable}). Hence its discriminant equals $\Delta(\Sc)=4c_2(\Sc)-c_1(\Sc)^2=3a^2$. So the restriction theorem of Langer \cite[Theorem 2.19]{langersurvey} tells us that $\Sc|_C$ is stable on every smooth curve $C$ of degree $n > \frac{3a^2+1}{2}$. But this bound grows quadratically with $a$ and therefore becomes expeditiously high.
\end{rem}

It was shown in \cite[Corollary 2]{miyaoka} that the set of primes where the bundle $\Syz_C(X^a,Y^a,Z^a)$ has strongly semistable reduction contains in general not almost all prime numbers disproving a stronger version of \cite[Problem 5.4]{defsss} which was conjectured by N. I. Shepherd-Barron in \cite{shepherdbarronsemistability}. That is, strong semistability is not an open property in arithmetic deformations. The following example shows that this phenomenon depends
on the pair $(a,n)$. But using Han's Theorem \ref{hansthm} and Corollary \ref{deltastrongsemistable} one can exhibit also this property explicitly.

\begin{exa}
\label{exampleforallmostallarithmetic}
Let $n=5$ and $p$ be an odd prime number with $\gcd(p,5)=1$. For all $e\in\N$ we can write $p^e=10\cdot l+r$ for some $l\in\N$ and $r\in\{1,3,7,9\}$. Then $\tfrac{p^e}{5}$ is 
$2l+\tfrac{1}{5}$, $2l+\tfrac{3}{5}$, $2l+1+\tfrac{2}{5}$ resp. $2l+1+\tfrac{4}{5}$ and hence the nearest element in $L_{\text{odd}}$ to 
$(\tfrac{p^e}{5},\tfrac{p^e}{5},\tfrac{p^e}{5})$ is given by $(2l+1,2l,2l)$, $(2l+1,2l+1,2l+1)$, $(2l+1,2l+1,2l+1)$ resp. $(2l+1,2l+2,2l+2)$. Hence, the 
taxicab distance of $(\tfrac{p^e}{5},\tfrac{p^e}{5},\tfrac{p^e}{5})$ to $L_{\text{odd}}$ is $\tfrac{6}{5}$ in any case, showing 
$\delta(\tfrac{1}{5},\tfrac{1}{5},\tfrac{1}{5})=0$. Therefore, $\Omega_{\Prim^2}|_C$ is strongly semistable in almost all characteristics and semistable 
in characteristic zero by Theorem \ref{fermatgenericfibersemistability}. The only exceptional prime numbers are $2$ and $5$.
\end{exa}

If $\Syz(X^a,Y^a,Z^a)$ is not semistable on the Fermat curve in characteristic $0$ then Theorem \ref{fermatgenericfibersemistability} allows to compute its Harder-Narasimhan filtration via reduction to positive characteristic. It turns out that the HN-filtration in characteristic $0$ coincides with the HN-filtration in positive characteristic $p$ (see Theorem \ref{almarfilt}) for almost all prime numbers.

\begin{thm}
\label{hnfiltrationcharzero}
Let $k$ be a field of characteristic $0$ and denote by $$C:= \Proj(k[X,Y,Z]/(X^n+Y^n+Z^n))$$ the Fermat curve of degree $n$ over $k$. Further, let $a\geq 1$ be an integer and write $a=nl+r$ with $0 \leq r < n$. Set $\Sc:=\Syz_C(X^a,Y^a,Z^a)$.
If $l$ is even and $r > \frac{2n}{3}$ or $l$ is odd and $r<\frac{n}{3}$, then $\Sc$ is not semistable and its HN-filtration is given
by the short exact sequence 
$$0 \lto \Oc_C(-m) \lto \Sc \lto \Oc_C(m-3a) \lto 0,$$
where $m$ is defined as in Theorem \ref{almarfilt} in dependence of $l$.
\end{thm}

\begin{proof}
By Lemma \ref{nonstablefrobenius} we have a non-trivial mapping $\Oc_C(-m) \lto \Sc$. We have to show that this mapping has no zeros on $C$. Since $C$ is already defined over $\Z$ we can reduce to the case $k=\Q$.
Assume that we have a factorization $$\Oc_C(-m) \lto\Lc \lto \Syz_C(X^a,Y^a,Z^a),$$ where $\Lc$ is the maximal destabilizing subbundle. As in Theorem \ref{fermatgenericfibersemistability} we consider $C$ as the generic fiber of the relative curve
$$\Cc = \Proj(\Z_n[X,Y,Z]/(X^n+Y^n+Z^n)) \lto \Spec \Z_n.$$
On every special fiber $\Cc_p$ satisfying $p>\max\{n-3,\frac{3a}{2n}\}$ the HN-filtration of the restriction equals $0 \subset \Oc_{\Cc_p}(-m) \subset \Syz_{\Cc_p}(X^a,Y^a,Z^a)$ by Theorem \ref{almarfilt}. Since the HN-filtration of $\Syz_C(X^a,Y^a,Z^a)$ extends to an open subset of $\Spec \Z_n$, we obtain $\Lc \cong \Oc_C(-m)$.
\end{proof}

\begin{exa}
\label{brennerkatzmananswer}
In this example we provide an application of our results to the theory of \emph{tight closure} which is related to Hilbert-Kunz theory and strongly connected to the theory of (strongly) semistable vector bundles due to the work of Brenner
(see \cite{brennerbarcelona}). Via this geometric approach Problem \ref{brennerfermatproblem} is related to Hochster's question \cite[Question 13]{hochstertightsolid}.
In the homogeneous coordinate ring of the Fermat septic ($n = 7$), Brenner and Katzman have shown in \cite[Theorem 4.1]{brennerkatzmanarithmetic} by tedious computations that $X^3Y^3 \in (X^4,Y^4,Z^4)^*$ (the tight closure) for prime numbers $p \equiv 3 \!\! \mod 7$ and $X^3Y^3 \notin (X^4,Y^4,Z^4)^*$ for $p \equiv 2 \!\! \mod 7$. That is, the relative curve $\Cc: \Proj(\Z[X,Y,Z]/(X^7+Y^7+Z^7)) \ra \Spec \Z$
illustrates that tight closure does not behave uniformly in the fibers of an arithmetic deformation. A similar reasoning as in Examples \ref{syzsquarestableexample} and \ref{exampleforallmostallarithmetic} shows that $\Syz_C(X^4,Y^4,Z^4)$ is strongly semistable on the Fermat septic if and only if $p \equiv \pm 1 \!\! \mod 7$. In particular, the monomial $X^3Y^3$ of degree $6$ belongs to the tight closure of the ideal $(X^4,Y^4,Z^4)$ in $k[X,Y,Z]/(X^7+Y^7+Z^7)$ in these characteristics by \cite[Theorem 6.4]{brennerbarcelona}. Hence, we easily get infinitely many prime numbers where the tight closure inclusion does hold. Furthermore, the question for which prime numbers the syzygy bundle $\Syz(X^4,Y^4,Z^4)$ has strongly semistable reduction, raised in  \cite[paragraph before Corollary 4.3]{brennerkatzmanarithmetic}, is now answered.
\end{exa}

\bibliographystyle{alpha} 
\bibliography{bibfile}

\end{document}